\documentclass[a4paper, 12pt]{amsart}
\usepackage{amssymb, amsfonts, amsmath, bm}
\usepackage[all]{xy}
\usepackage[shortlabels]{enumitem}
\usepackage{hyperref, color}
\usepackage{commath}
\usepackage{esdiff}
\usepackage{graphicx,subfigure}
\usepackage[center]{caption}
\usepackage{amsthm}
\usepackage{diagbox}
\usepackage{multirow}
\usepackage{esdiff}
\usepackage[pagewise]{lineno}
\usepackage[utf8]{inputenc}
\usepackage{soul}
\usepackage{comment} 
\usepackage{hyperref, color}
\hypersetup{
	colorlinks = true,
	linkcolor = blue,
	filecolor = blue,
	urlcolor = red,
	citecolor = blue
}
\usepackage{cite}

\usepackage[nameinlink,capitalise]{cleveref}

\usepackage[a4paper,twoside,top=1in, bottom=1in, left=0.8in, right=0.8in]{geometry}

\DeclareMathOperator{\divv}{div}

\newcommand{\Rbb}{\mathbb{R}}

\newcommand{\tg}{\tilde{g}}
\newcommand{\tX}{\widetilde{X}}
\newcommand{\talpha}{\tilde{\alpha}}
\newcommand{\tv}{\tilde{v}}

\newcommand{\tnabla}{\widetilde{\nabla}}
\newcommand{\tH}{\widetilde{H}}
\newcommand{\tdelta}{\tilde{\delta}}
\newcommand{\ta}{\tilde{a}}
\newcommand{\tmu}{\tilde{\mu}}
\newcommand{\tE}{\tilde{E}}

\newtheorem{theorem}{Theorem}[]
\newtheorem{lemma}[theorem]{Lemma}
\newtheorem{proposition}[theorem]{Proposition}

\newtheorem*{conjecture*}{Conjecture}

\newtheorem{theo}{Theorem}

\newtheorem{cor}[theo]{Corollary}

\theoremstyle{definition}

\newtheorem*{remark*}{Remark}

\newtheorem{remark}[theorem]{Remark}

\numberwithin{figure}{section}

\usepackage{graphicx} 

\title[On the stability of free boundary minimal submanifolds in conformal domains]{On the stability of free boundary minimal submanifolds in conformal domains}

\author[Alcides de Carvalho, Roney Santos and Federico Trinca]{Alcides de Carvalho, Roney Santos and Federico Trinca}

        \address{Universidade Federal de Alagoas\\
		Instituto de Matem\'atica\\
		Campus A. C. Sim\~oes, BR 104 - Norte, Km 97, 57072-970, Macei\'o - AL, Brazil.}	
	\email{alcides.junior@im.ufal.br}
 
	\address{Universidade de S\~ao Paulo\\
		Departamento de Matem\'atica\\
        Rua do Mat\~ao, 05508-900, S\~ao Paulo - SP, Brazil.}
	\email{roneypsantos@ime.usp.br}
 
         \address{Department of Mathematics\\ UBC \\ 1984 Mathematics Road, Vancouver Canada.}
            \email{ftrinca@math.ubc.ca}

\begin{document}

\begin{abstract}
    Given a $n$-dimensional Riemannian manifold with non-negative sectional curvatures and convex boundary, that is conformal to an Euclidean convex bounded domain, we show that it does not contain any compact stable free boundary minimal submanifold of dimension $2\leq k\leq n-2$, provided that either the boundary is strictly convex with respect to any of the two metrics or the sectional curvatures are strictly positive.
\end{abstract}

\maketitle

\section*{Introduction}
A fundamental topic in differential geometry is the study of submanifolds that are critical points of some functional within a given Riemannian manifold. The classical examples of such objects are \emph{minimal submanifolds}, i.e., the critical points of the volume functional. Indeed, minimal submanifolds have been studied since 1760 and remain among the most active research fields in geometry, with several applications to topology and mathematical physics.

In recent years, when the compact Riemannian manifold $M$ has non-empty boundary $\partial M$, critical points of the volume functional among submanifolds whose non-empty boundaries lie in $\partial M$ attracted particular attention. Such objects are called \emph{free boundary minimal submanifolds} and are characterized by having vanishing mean curvature and by meeting $\partial M$ orthogonally along their boundary. Being variational objects, it is also natural to look at the second derivative of the functional and, in particular, at the free boundary minimal submanifolds with nonnegative second derivative, which are called \emph{stable}. The free boundary minimal submanifolds that are not stable are called \emph{unstable}. 

In this context, it is well-known that the positivity of the Riemann curvature tensor of $M$ and the convexity of $\partial M$, either individually or together, imply that free boundary minimal submanifolds must be unstable. For instance, the classical work of Simons \cite{Simons1968} implies that in a $p$-convex domain of the round sphere, any $k$-dimensional free boundary minimal submanifold is unstable if $k\leq n-p$. More generally, one can deduce from \cite[Proposition 2]{LawsonSimons} and \cite[Theorem 2]{ShenXu2000} that the same holds true for Riemannian manifolds whose sectional curvatures lie in the interval $(\frac{1}{4},1]$, called {\it $\frac 1 4$-pinched manifolds}, that have $p$-convex boundary and that are isometrically immersed as hypersurfaces of some Euclidean space (cfr. \cite{HuWei2003, ShenHe2001} for improved results in this direction). Analogously, Fraser \cite{Fraser2000,FraserIndex} proved the same conclusion for strictly $p$-convex Euclidean domains.
The cases of geodesics, $2$-dimensional disks and hypersurfaces are better understood, where Morse index estimates are known for more general ambient manifolds (cfr. \cite{AmbrozioCarlottoSharp2018,MisteliFranz2024,Fraser2000,FraserIndex,FraserLi,Lima2022,LimaMenezes,Medvedev,MooreSchulteIndex,Sargent2017,Schoen} and references therein).

In this work, we generalise both Simons' result on domains of the round sphere and Fraser's on Euclidean domains as follows.
\begin{theo}\label{maintheorem}
    Let $(\Omega,\tg)$ be a $n$-dimensional Riemannian manifold with $p$-convex boundary and non-negative sectional curvatures that is conformal to a $p$-convex bounded domain $(\Omega,g)$ of $\Rbb^n$. If either:
    \begin{enumerate}
        \item[$\mathrm{(i)}$] $(\Omega,\tg)$ has positive sectional curvature;
        \item[$\mathrm{(ii)}$] $\partial\Omega$ is strictly $p$-convex with respect to $g$;
        \item[$\mathrm{(iii)}$] $\partial\Omega$ is strictly $p$-convex with respect to $\tg$,
    \end{enumerate}
    then $(\Omega, \tg)$ does not contain $k$-dimensional compact stable free boundary minimal submanifolds for all $2\leq k \leq \min\{n-2,n-p\}.$
\end{theo}

\begin{remark*}
A priori, one could not assume that free boundary minimal submanifold have non-empty boundary. In this case, \cref{maintheorem} under the assumption $\mathrm{(i)}$ still holds true.
\end{remark*}

\begin{remark*}
    Note that under the assumptions of \cref{maintheorem}, the case of geodesics (with non-empty boundary) follows from the work of Frankel (cfr. \cite[E3 Frankel]{Schoen}). The stability inequality takes care of the case of oriented hypersurfaces (cfr. \cite[Lemma 2.1]{FraserLi}). We stated \cref{maintheorem} for submanifolds of dimension between $2$ and $n-2$ because our argument is only valid for these dimensions. We will emphasize in the proof of \cref{inequality} where this dimensional condition is essential. 
\end{remark*}

From \cref{maintheorem}, we deduce:
\begin{cor}\label{maincor}
   Let $(\Omega,\tg)$ be a $n$-dimensional Riemannian manifold with non-negative sectional curvatures that is conformal to a bounded domain $(\Omega,g)$ of $\Rbb^n$ under the conformal transformation $\tg=e^{2u}g$. If either:
   \begin{enumerate}
       \item[$\mathrm{(i)}$]$\partial \Omega$ is $p$-convex with respect to $g$ and $u$ is a strictly increasing function in the exterior direction of $\partial \Omega$;
        \item[$\mathrm{(ii)}$] $\partial \Omega$ is $p$-convex with respect to $\tg$ and $u$ is a strictly decreasing function in the exterior direction of $\partial \Omega$,
   \end{enumerate}
   then $(\Omega, \tg)$ does not contain $k$-dimensional compact stable free boundary minimal submanifolds for all $2\leq k \leq \min\{n-2,n-p\}.$
\end{cor}

The results of the present paper can also be interpreted in light of the (elusive) analogy between free boundary minimal submanifolds of the Euclidean ball and closed minimal submanifolds of the round sphere (cfr. \cite{FraserSchoen2011,FraserSchoenSharp,LiFreeBoundary}). Indeed, \cref{maintheorem} can be viewed as the free boundary counterpart to the result of Giada Franz and the third author \cite{FranzTrinca}, along with the subsequent improvement by Hang Chen \cite{chen2023conformal}, which demonstrates that closed minimal submanifolds in positively curved spheres that are conformal to the round one are unstable. The idea used there is to trace the second variation over the family of conformally rescaled constant vector fields projected to the sphere. Our technique is directly inspired by this, indeed, the role of conformally rescaled constant vector fields projected to the sphere is played by conformally rescaled constant vector fields of the Euclidean space (cfr. \cite[Theorem 5.1.1]{Simons1968} and \cite[Theorem 2.2]{Schoen}).

It is interesting to notice that \cite{chen2023conformal,FranzTrinca,HuWei2003,ShenHe2001,ShenXu2000} fit in the context of the celebrated Lawson--Simons conjecture \cite{LawsonSimons}, which proposes that minimal submanifolds of closed, simply-connected, $\frac 1 4$-pinched Riemannian manifolds must be unstable. 
Observe that $\frac 1 4$-pinched is motivated by the well-known fact that $\mathbb C\mathbb P^n$ endowed with the Fubini--Study metric is a compact, simply-connected manifold with curvature between $[\frac 1 4,1]$ admitting stable minimal submanifolds, i.e., the complex submanifolds. 

Given the aforementioned analogy between closed minimal submanifolds in the round sphere and free boundary minimal submanifolds in the unit ball, it would be natural to propose a free boundary version of the Lawson--Simons conjecture with some version of \cref{maintheorem} as a corroborating evidence. Unfortunately, it is unclear to the authors what should be the correct curvature and boundary conditions in this context. Indeed, the standard calibration argument proving the stability of complex submanifolds in K\"ahler manifolds does not automatically go through when the submanifold has non-empty boundary. For instance, it is straightforward to observe from \cite[Theorem 6]{LawsonSimons} that free boundary minimal submanifolds of strictly convex domains of $\mathbb C\mathbb P^n$ are unstable. On the other hand, if we allow free boundary minimal submanifolds to have empty boundary and we allow weaker convexity assumptions on the boundary of the ambient manifold, then one can find closed complex submanifolds (hence, stable) in $p$-convex domains of $\mathbb C\mathbb P^n$ with $p\geq 3$ (cfr. \cite{Takagi1975} and the excellent reference \cite[p. 351-352]{CecilRyan2015}).

The paper is organized as follows. In \cref{preliminaries}, we fix our notations and we recall classical notions on conformal geometry and free boundary minimal submanifolds. In \cref{secondvariation}, we study the behavior of a certain quadratic form under conformal transformations. This quadratic form coincides with the second variation formula in the case of minimal submanifolds. In \cref{proof}, we prove \cref{maintheorem} and \cref{maincor}.
\subsection*{Acknowledgements}
The authors would like to thank Giada Franz and Celso Viana for many helpful comments and discussions. Federico Trinca would also like to express his gratitude to Alessandro Carlotto for introducing him to the work of Lawson--Simons and the study of free boundary minimal submanifolds, and to Jesse Madnick for pointing out the reference \cite{CecilRyan2015}. Alcides de Carvalho was partially supported by CNPq grant 150285/2023-0 as well as FAPEAL under the process E:60030.0000000513/2023. Roney Santos was partially supported by grant 2023/14796-3, S\~ao Paulo Research Foundation (FAPESP). Federico Trinca was partially supported by the Royal Society University Research Fellowship Renewal 2022 URF\textbackslash R\textbackslash 221030 and by the Pacific Institute for the Mathematical Sciences (PIMS).

\section{Preliminaries}\label{preliminaries}

Let $(M,g)$ be a compact $n$-dimensional Riemannian manifold of Levi-Civita connection $\nabla$. 

\subsection{Curvature} 
We assume (by convention) the Riemann curvature tensor $R_g$ of $(M,g)$ to be given by
\[R_g(X,Y)Z = \nabla_Y\nabla_XZ - \nabla_X\nabla_YZ + \nabla_{[X,Y]}Z,\quad \mbox{$X,Y,Z \in \mathfrak{X}(M).$}\]
For any $x\in M$ and two-plane $\Pi\subset T_x M$, we denote by $K_g(\Pi)$ the sectional curvature of $\Pi$, with respect to $g$. If $\Pi$ is generated by the vectors $X,Y\in T_x M$, then $K_g(\Pi)$ can also be indicated by $K_g(X,Y)$, and have the form
\[K_g(X,Y) = \frac{g(R_g(X,Y)X,Y)}{|X|^2_g|Y|^2_g - g(X,Y)},\]
where $|\cdot|_g$ is the norm induced by $g.$

\subsection{Submanifold theory}

Let $\Sigma$ be a $k$-dimensional submanifold of $M$, and let $T\Sigma$ and $N\Sigma$ be the tangent bundle and the normal bundle of $\Sigma$, respectively. The tangent and the normal bundles naturally admit connections compatible with the Levi-Civita connection $\nabla$ of $M$, which we denote by $\nabla^\top$ and $\nabla^\perp$, respectively. Such connections satisfy $\nabla^\top_X Y=(\nabla_X Y)^\top$ and $\nabla^\perp_X V=(\nabla_X V)^\perp$ for all smooth sections $X,Y \in \mathfrak{X}(\Sigma)$ and $V \in \Gamma(N\Sigma)$. 

The divergence operator $\divv_\Sigma$ is defined by $\divv_\Sigma(X) = \sum_{i=1}^k g(\nabla_{v_i}X,v_i)$, where $\{v_1, \ldots, v_k\}$ is a local orthonormal frame of $\Sigma.$

Let $u: M \to \Rbb$ be a smooth function. We denote the gradient of $u$ with respect to $g$ by $\nabla u$ and its Hessian by $\nabla^2u,$ which is defined, for every $X,Y \in \mathfrak{X}(M)$, as $\nabla^2u(X,Y) = g(\nabla_X\nabla u, Y)$. We indicate the tangent and normal parts of $\nabla u$ with respect to $\Sigma$ as $\nabla^\top u = (\nabla u)^\top$ and $\nabla^\perp u = (\nabla u)^\perp$, respectively.

Recall that the second fundamental form $\alpha$ of the submanifold $\Sigma$ is the symmetric tensor defined by
$$\alpha(X,Y) = (\nabla_X Y)^\perp,$$
for all $X,Y \in \mathfrak{X}(\Sigma)$. The trace of $\alpha$ gives the mean curvature $H$ of $\Sigma$. This means that 
\[H(x) = \sum_{i=1}^k\alpha(v_i,v_i),\]
for an arbitrary orthonormal basis $\{v_1, \ldots, v_k\}$ of $T_x\Sigma.$ If $\Sigma$ is compact, we can compute its $k$-volume by
\[vol_k(\Sigma) = \int_\Sigma d\mu,\]
where $d\mu$ is the volume element induced on $\Sigma$ by $g.$

We now discuss the geometric condition we want to impose on $\partial M$, which we see as a hypersurface in $M$ oriented by its inward-pointing normal $\eta$. For every integer $1\leq p\leq n-1$, we say that $\partial M$ is {\it $p$-convex} if the sum of the lowest $p$ principal curvatures is non-negative at each point of $\partial M$. Equivalently, this means that at each $x\in \partial M$ for every orthonormal vectors $\{v_1,\dots,v_p\}$ we have
\[
\sum_{i=1}^p g(\alpha_{\partial M}(v_i,v_i),\eta) \geq 0,
\]
where $\alpha_{\partial M}$ is the second fundamental form of $\partial M$. If the inequality is strict, we say that $\partial M$ is {\it strictly $p$-convex}. The special cases of $p=1$ or $p=n-1$ correspond to \emph{convexity} and \emph{mean-convexity}, respectively. 

\subsection{Conformal geometry}

We say that two Riemannian metrics $\tg$ and $g$ on $M$ are \textit{conformal} if there is a smooth function $u: M \to \Rbb$ such that $\tg = e^{2u}g.$ In this case, the function $e^{2u}$ is called \textit{conformal factor}. The following results collect some classical formulas relating objects taken with respect to $\tg$ and $g$ in terms of $u.$

\begin{proposition}[cf. {\cite[Theorem 1.159]{Besse}}]\label{conformallemma}
    Let $M$ be a $n$-dimensional manifold, and let $\tg$ and $g$ be conformal metrics on $M$ with conformal factor $e^{2u}.$ The following identities holds for vector fields $X,Y,Z \in \mathfrak{X}(M).$
    \begin{itemize}
        \item[$\mathrm{(i)}$]\label{item 1} The Levi-Civita connections $\tnabla$ and $\nabla$ with respect to $\tg$ and $g$ are related by
        \[\tnabla_XY = \nabla_XY + X(u)Y + Y(u)X - g(X,Y)\nabla u.\]
        
        \item[$\mathrm{(ii)}$]\label{item 2} The Riemann curvature tensors $R_{\tg}$ and $R_g$ respectively of $\tg$ and $g$ are related by
        \begin{align*}
            R_{\tg}(X,Y)Z &= R_g(X,Y)Z + X(u)Z(u)Y - Y(u)Z(u)X - X(u)g(Y,Z)\nabla u\\
            &\quad + Y(u)g(X,Z)\nabla u - g(X,Z)\nabla_Y\nabla u + g(Y,Z)\nabla_X\nabla u\\
            &\quad - g(X,Z)|\nabla u|^2_gY + g(Y,Z)|\nabla u|^2_gX - \nabla^2u(X,Z)Y + \nabla^2u(Y,Z)X.
        \end{align*}
        
        \item[$\mathrm{(iii)}$]\label{item 3} The sectional curvatures $K_{\tg}$ and $K_g$ with respect to $\tg$ and $g$ are related by
        \[e^{2u}K_{\tg}(X,Y) = K_g(X,Y) + X(u)^2 + Y(u)^2 - |\nabla u|^2_g - \nabla^2u(X,X) - \nabla^2u(Y,Y),\]
        for any $X$ and $Y$ orthonormal with respect to $g$. 
        \item[$\mathrm{(iv)}$]\label{item 4} The volume elements $d\tmu$ and $d\mu$ with respect to $\tg$ and $g$ are related by
        \[d\tmu = e^{nu}d\mu.\]
    \end{itemize}
\end{proposition}

Using the formula for $\tnabla$ and $\nabla,$ it is straightforward to conclude the following.

\begin{proposition}[cf. {\cite[Eq. 1.163]{Besse}}]\label{conformallemma2}
    Let $M$ be a manifold and let $\Sigma$ be a $k$-dimensional submanifold of $M$. If $\tg$ and $g$ are conformal metrics on $M$ with conformal factor $e^{2u}$, then the second fundamental forms $\talpha$ and $\alpha$ of $\Sigma$ with respect to $\tg$ and $g$, respectively, are related by
    \[\talpha(X,Y) = \alpha(X,Y) - g(X,Y)\nabla^\perp u\]
    for all $X,Y \in \mathfrak{X}(\Sigma).$ In particular, the mean curvature vectors $\tH$ and $H$ of $\Sigma$ with respect to $\tg$ and $g$, respectively, are related by
        \[e^{2u}\tH = H - k\nabla^\perp u.\]
\end{proposition}

\subsection{Free boundary minimal submanifolds}
We now assume that $\partial M \neq \emptyset$, and we consider $f: \Sigma \to M$ to be a compact $k$-dimensional submanifold of $M$ such that $f( \Sigma)\cap \partial M= f(\partial \Sigma)$ and $\partial\Sigma\neq\emptyset$. In this context, a {\it variation} of $\Sigma$ is a map $F: (-\varepsilon, \varepsilon) \times \Sigma
\rightarrow M$ satisfying:
\begin{itemize}
    \item[$\mathrm{(i)}$] $F$ is smooth;
    \item[$\mathrm{(ii)}$] $F(0,x) = f(x)$ for all $x \in \Sigma$;
    \item[$\mathrm{(iii)}$] $F(t,y)\in \partial M$ for all $y \in \partial\Sigma.$
\end{itemize}
Every variation $F$ induces a {\it variational vector field} $X$, which is the vector field defined, for all $x \in \Sigma$, by 
\[X(x) = \frac{d}{dt}f_t(x)\Big|_{t=0}, \]
where $f_t(x) = F(t,x)$. Obviously, $X$ is tangent to $\partial M$ along $f(\partial\Sigma)$. Conversely, from any vector field which is  tangent to $\partial M$ along $f(\partial\Sigma)$ one can construct a variation via the exponential map. In order to keep our notation light, we will identify $\Sigma$ and $\partial\Sigma$ with their images $f(\Sigma)$ and $f(\partial\Sigma)$ from now on.

In the usual way, for any variational vector field $X$ as above, we can define the first and second variation of the volume of $\Sigma$ as follows
\[\delta\Sigma(X) = \frac{d}{dt}vol_k(f_t(\Sigma))\Big|_{t=0}\quad \mbox{and}\quad \delta^2\Sigma(X,X) = \frac{d^2}{dt^2}vol_k(f_t(\Sigma))\Big|_{t=0}.\]

A \textit{free boundary minimal submanifold} is a critical point of the volume functional, i.e., $\delta\Sigma(X) = 0$ for every variational vector field $X$. In our setting, the classical first variation formula (cfr. \cite[Section 2]{Schoen}) reads: 
\[\delta\Sigma(X) = -\int_\Sigma g(X, H)d\mu + \int_{\partial\Sigma} g(X, \nu_g)da,\]
where $\nu_g$ is the outward unit conormal of $\Sigma$ (i.e., $\nu_g$ is the unique outer unit vector tangent to $\Sigma$ and normal to $\partial\Sigma$), $d\mu$ and $da$ are the volume elements induced by $g$ in $\Sigma$ and $\partial\Sigma,$ respectively. Hence, $\Sigma$ is a free boundary minimal submanifold if and only if $H\equiv0$ and $g(X,\nu_g) \equiv 0$ for every $X$ tangent to $\partial M$ along $\partial\Sigma$. The latter condition is commonly referred to as the {\it free boundary condition}, and geometrically means that $\Sigma$ meets $\partial M$ orthogonally along $\partial \Sigma$.

A free boundary minimal submanifold $\Sigma$ is \textit{stable} if $\delta^2\Sigma(X,X) \geq 0$ for all $X\in\Gamma(N\Sigma),$ and \textit{unstable} if there exists a $X\in\Gamma(N\Sigma)$ such that $\delta^2\Sigma(X,X) < 0.$ In this setting, the classical second variation formula proves that $\delta^2\Sigma(X,X)$ admits the following simpler form (cfr. \cite[Section 2]{Schoen}):
\[\delta^2\Sigma(X, X) = \int_\Sigma\bigg(|\nabla^\perp X|^2_g - \sum_{i=1}^kg(R_g(X,v_i)X,v_i) - g(\alpha,X)^2\bigg)d\mu + \int_{\partial\Sigma} g(\nabla_XX, \nu_g) da,\]
 where  $\{v_1, \ldots, v_k\}$ is an arbitrary orthonormal frame of $\Sigma,$ while
\[|\nabla^\perp X|^2_g = \sum_{i=1}^k|\nabla^\perp_{v_i}X|^2_g \quad \mbox{and}\quad g(\alpha,X)^2 = \sum_{i,j=1}^kg(\alpha(v_i,v_j),X)^2.\]

\section{The second variation of the free boundary volume functional after a conformal change of metric}\label{secondvariation}
Let $(M,g)$ be a compact $n$-dimensional Riemannian manifold with $\partial M\neq 0$, and let $\Sigma$ be a $k$-dimensional free boundary submanifold of $M$ such that $\Sigma\cap \partial M= \partial \Sigma$ (not necessarily minimal). We consider the quadratic operator defined for every $X\in \Gamma(N\Sigma)$ by 
\[Q_g(X,X) = \int_\Sigma S_g(X,X)d\mu + \int_{\partial\Sigma}T_g(X,X)da,\]
where, resepectively at each point of $\Sigma$ and $\partial\Sigma$, $S_g$ and $T_g$ are the quadratic operators given by
\[S_g(X,X) = |\nabla^\perp X|^2_g - \sum_{i=1}^kg(R_g(X,v_i)X,v_i) - g(\alpha,X)^2\quad \mbox{and}\quad T_g(X,X) = g(\nabla_XX,\nu_g).\]
As usual, $\{v_1,\dots,v_k\}$ represents a local orthonormal frame of $\Sigma$ and $\nu_g$ is the outward unit conormal of $\Sigma$. Observe that since $X$ is normal to $\Sigma$, the free boundary condition guarantees that it is tangent to $\partial M$ along $\partial\Sigma$. Also, if $\Sigma$ is a compact free boundary minimal submanifold, the operator $Q_g$ we have just defined coincides with the stability operator $\delta^2\Sigma$. 

Let $\tilde{g}$ be a metric conformal to $g$ of conformal factor $e^{2u}$. Similarly to \cite{FranzTrinca}, we now explain how $Q_{\tilde{g}}$ (in particular, $S_{\tilde{g}}$ and $T_{\tilde{g}}$) can be expressed in terms of $Q_g$ (in particular, $S_g$ and $T_g$). Note that $\Gamma(N\Sigma)$ and the free boundary condition are preserved under conformal change of metrics, hence, it makes sense to consider $Q_{\tilde{g}}$ and $Q_g$ for the given submanifold $\Sigma$. The formulas for the interior terms $S_g$ and $S_{\tilde g}$ can be recovered from \cite[Proposition 4.4 and Corollary 4.5]{FranzTrinca}. For the sake of completeness, we include the computations here as well.

\begin{lemma}\label{Q}
    Let $(M,g)$ be a compact $n$-dimensional Riemannian manifold, and let $\tg$ be a Riemannian metric conformal to $g$ with conformal factor $e^{2u}$. If $\Sigma$ is a $k$-dimensional free boundary minimal submanifold of $(M,\tg)$, then
    $$S_{\tg}(X,X) = S_g(X,X) + (\nabla^\top u)(|X|^2_g) + |X|^2_g\divv_\Sigma(\nabla u) + k|X|^2_g|\nabla u|^2_g + k\nabla^2u(X,X),$$
    $$T_{\tg}(X,X) =  e^u(T_g(X,X) - |X|^2_g\nu_g(u))$$
    for each vector field $X\in\Gamma(N\Sigma)$. In particular, if $\tX = e^{-u}X$, then 
    $$e^{2u}S_{\tg}(\tX,\tX) = S_g(X,X) - |X|^2_g|\nabla^\top u|^2_g + |X|^2_g\divv_\Sigma(\nabla u) + k|X|^2_g|\nabla u|^2_g + k\nabla^2u(X,X),$$
    $$e^uT_{\tg}(\tX,\tX) = T_g(X,X) - |X|^2_g\nu_g(u).$$
\end{lemma}

\begin{proof}
    Fix $\{\tv_1, \ldots, \tv_k\}$ a local orthonormal frame of $\Sigma$ with respect to $\tg$, which induces the local orthonormal frame of $\Sigma$ with respect to $g$: $\{v_1, \ldots, v_k\}$, where $v_i=e^u \tv_i$ for all $i = 1, \ldots, k$.

    We first deal with $S_{\tg}(X,X)$, which we recall to be
    \[S_{\tg}(X,X) = |\tnabla^\perp X|^2_{\tg} - \sum_{i=1}^k\tg(R_{\tg}(X,\tv_i)X,\tv_i) - \tg(\talpha,X)^2.\]
    Here, $\talpha$ denotes the second fundamental form of $\Sigma$ with respect to $\tg.$ Since (cfr. Item (i) in \cref{conformallemma})
    \begin{align*}
       \tnabla^\perp_{\tv_i}X = e^{-u}\big(\nabla_{v_i}X + v_i(u)X\big)^\perp = e^{-u}\big(\nabla^\perp_{v_i} X + v_i(u)X\big),
    \end{align*}
    we obtain the following formula for the first term of $S_{\tg}(X,X)$:
    \begin{align*}
        |\tnabla^\perp X|^2_{\tg} &= \sum_{i=1}^kg(\nabla^\perp_{v_i} X, \nabla^\perp_{v_i} X) + 2\sum_{i=1}^kv_i(u)g(\nabla^\perp_{v_i} X,X) + \sum_{i=1}^kv_i(u)^2g(X,X)\\
        &= |\nabla^\perp X|^2_g + (\nabla^\top u)(|X|^2_g) + |X|^2_g|\nabla^\top u|^2_g.
    \end{align*}
    The second term can be rewritten in the following form using Item (ii) in \cref{conformallemma}
    \begin{align*}
        \sum_{i=1}^k\tg(R_{\tg}(X,\tv_i)X,\tv_i) &= \sum_{i=1}^k \left(\tg(R_g(X,\tv_i)X,\tv_i)+ X(u)^2 + \tv_i(u)|X|^2_g\tg(\nabla u,\tv_i)\right)\\
        &\quad \quad \quad \quad - \sum_{i=1}^k\left(|X|^2_g\tg(\nabla_{\tv_i}\nabla u,\tv_i) + |X|^2_g|\nabla u|^2_g + \nabla^2u(X,X)\right)\\
        &= \sum_{i=1}^k g(R_g(X,v_i)X,v_i) + kX(u)^2 + |X|^2_g|\nabla^\top u|^2_g\\
        &\quad \quad \quad \quad - |X|^2_g\divv_\Sigma(\nabla u) - k|X|^2_g|\nabla u|^2_g - k\nabla^2u(X,X).
    \end{align*}
    For the last term, we first use \cref{conformallemma2} to compute
    \begin{align*}
        \tg(\talpha(\tv_i,\tv_j),X) = g(\alpha(v_i,v_j) - g(v_i,v_j)\nabla^\perp u,X) = g(\alpha(v_i,v_j),X) - g(v_i,v_j)X(u),
    \end{align*}
    which also implies
    \begin{align*}
        \tg(\talpha,X)^2 &= \sum_{i,j=1}^k(g(\alpha(v_i,v_j),X) - g(v_i,v_j)X(u))^2\\
        &= \sum_{i,j=1}^kg(\alpha(v_i,v_j),X)^2 - 2X(u)\sum_{i=1}^kg(\alpha(v_i,v_i),X) + X(u)^2\sum_{i=1}^kg(v_i,v_i)^2\\
        &= g(\alpha,X)^2 - 2X(u)g(X,H) + kX(u)^2,
    \end{align*}
    where $H$ is the mean curvature of $\Sigma$ with respect to $g.$ The minimality of $\Sigma$ with respect to $\tg$ implies (cfr. \cref{conformallemma2}) $g(X,H) = kX(u)$, from which we conclude that:
    \[\tg(\talpha,X)^2 = g(\alpha,X)^2 - kX(u)^2.\]
    Replacing such identities in $S_{\tg}$ we obtain the desired formula:
     \[S_{\tg}(X,X) = S_g(X,X) + (\nabla^\top u)(|X|^2_g) + |X|^2_g\divv_\Sigma(\nabla u) + k|X|^2_g|\nabla u|^2_g + k\nabla^2u(X,X).\]

    When $\tX = e^{-u}X,$ the only non-tensorial terms in the previous formula to adjust are $(\nabla^\top u)(|\tX|^2_g)$ and $|\nabla^\perp\tX|^2_g.$ The former transforms as:
    \[(\nabla^\top u)(|\tX|^2_g) = (\nabla^\top u)(e^{-2u}|X|^2_g) = e^{-2u}\left((\nabla^\top u)(|X|^2_g) - 2|X|^2_g|\nabla^\top u|^2_g\right),\]
    while the latter as:
    \begin{align*}
        |\nabla^\perp\tX|_g^2 &= \sum_{i=1}^k|\nabla^\perp_{v_i}(e^{-u}X)|_g^2\\
        &= \sum_{i=1}^k|e^{-u}\nabla^\perp_{v_i}X - e^{-u}v_i(u)X|^2_g\\
        &= e^{-2u}\sum_{i=1}^k\left(|\nabla^\perp_{v_i} X|_g^2 - 2v_i(u)g(\nabla^\perp_{v_i}X,X) + |X|^2_gv_i(u)^2\right)\\
        &= e^{-2u}\left(|\nabla^\perp X|^2_g - (\nabla^\top u)(|X|^2_g) + |X|^2_g|\nabla^\top u|^2_g\right).
    \end{align*}
    
    We now turn our attention to the boundary term $T_{\tg}(X,X) = \tg(\tnabla_XX,\nu_{\tg})$. Since $\nu_{\tg} = e^{-u}\nu_g,$ we can use Item (i) of \cref{conformallemma} to deduce:
    \begin{align*}
        T_{\tg}(X,X) &= \tg(\tnabla_XX, \nu_{\tg})\\
        &= e^ug(\nabla_XX + 2X(u)X - |X|^2_g\nabla u, \nu_g)\\
        &= e^u(T_g(X,X) - |X|^2_g\nu_g(u)).
    \end{align*}
    Since $T_{\tg}$ is tensorial, it is straightforward to compute $T_{\tg}(\tX,\tX)$.
    \end{proof}

\section{Proof of the main results}

In this section we prove \cref{maintheorem} and \cref{maincor}. Hence, from now on, $(\Omega,g)$ will denote a bounded domain of $\mathbb{R}^n$ endowed with the Euclidean metric and $\tg$ will be a metric conformal to $g$ of conformal factor $e^{2u}$.

As a first step, for \emph{any} $k$-dimensional free boundary submanifold $\Sigma$ of $\Omega$ (not necessarily minimal), we compute the trace of the quadratic operator $Q_g$, restricted to constant vector fields of $\mathbb{R}^{n}$ projected to the normal bundle of $\Sigma$. The free boundary condition guarantees that such vector fields are in $\Gamma (N\Sigma)$, i.e., are tangential to $\partial\Omega$ along $\partial \Sigma$ and so admissible variational vector fields for our problem.

\begin{lemma}\label{sum}
    Let $\Sigma$ be a $k$-dimensional free boundary submanifold of $(\Omega,g)$. Then at each point of $\Sigma$ and $\partial\Sigma,$ respectively, we have
    \[\sum_{\ell=1}^n S_g(E_\ell^\perp,E_\ell^\perp) = 0\quad \mbox{and}\quad \sum_{\ell=1}^n T_g(E_\ell^\perp,E_\ell^\perp) =  \sum_{\ell=1}^ng(\alpha_{\partial \Omega}(E_\ell^\perp,E_\ell^\perp), \nu_g),\]
   where $\{E_1, \ldots, E_n\}$ is an arbitrary orthonormal basis of $\Rbb^n$ and $\alpha_{\partial\Omega}$ is the second fundamental form of $\partial\Omega$ in $\Rbb^n.$
\end{lemma}

\begin{proof}
    Given $x\in \Sigma,$ let $\{v_1, \ldots, v_n\}$ be a local orthonormal frame such that $v_1, \ldots, v_k$ are tangent to $\Sigma$ and $v_{k+1}, \ldots, v_n$ are normal to $\Sigma.$ Since $\sum_{\ell=1}^nS_g(E_\ell^\perp,E_\ell^\perp)$ and $\sum_{\ell=1}^nT_g(E_\ell^\perp,E_\ell^\perp)$
    are independent of the orthonormal basis $\{E_1,\dots,E_n\}$, we can assume that $E_\ell(x) = v_\ell(x)$ for all $\ell=1,\dots,n$.
    
    For any constant vector field $E \in \Rbb^n$ and $v_i$ as above with $i=1,\ldots,k$, it holds
    \[\nabla^\perp_{v_i} E^\perp = (\nabla_{v_i} E - \nabla_{v_i} E^\top)^\perp = - \alpha(v_i, E^\top).\]
    From this, since the Riemann curvature tensor of $\Rbb^n$ is zero, we obtain that
    \[S_g(E^\perp,E^\perp) = |\nabla^\perp E^\perp|_g^2 - g(\alpha,E^\perp)^2 = \sum_{i=1}^k|\alpha(v_i,E^\top)|_g^2 - \sum_{i,j=1}^k g(\alpha(v_i,v_j),E^\perp)^2.\]
    Therefore, at $x$, the trace for $S_g$ becomes:
    \begin{align*}
        \sum_{\ell=1}^nS_g(E_\ell^\perp,E_\ell^\perp) &= \sum_{\ell=1}^n\sum_{i=1}^k|\alpha(v_i,v_\ell^\top)|_g^2 - \sum_{\ell=1}^n\sum_{i,j=1}^k g(\alpha(v_i,v_j),v_\ell^\perp)^2\\
        &= \sum_{i,j=1}^k|\alpha(v_i,v_j)|_g^2 - \sum_{i,j=1}^k|\alpha(v_i,v_j)|_g^2 = 0.
    \end{align*}
    
    We now consider $x \in \partial\Omega.$ By the free boundary condition, we must have that
    \[T_g(E^\perp,E^\perp) = g(\nabla_{E^\perp}E^\perp,\nu_g) = g(\alpha_{\partial \Omega}(E^\perp,E^\perp), \nu_g),\]
    for any contant vector field $E\in\Rbb^n.$ In particular,
    \[\sum_{\ell=1}^nT_g(E_\ell^\perp,E_\ell^\perp) = \sum_{\ell=1}^ng(\alpha_{\partial \Omega}(E_\ell^\perp,E_\ell^\perp), \nu_g),\]
    and the result follows.
\end{proof}

\begin{remark}
    By \cite{FraserSchoenMinimal, FraserSchoenSharp}, for any constant vector field $E\in \Rbb^n,$ the second variation of a compact $k$-dimensional free boundary minimal submanifold $\Sigma$ of the Euclidean unit ball $B^n$ in the direction of $E^\perp$ is
\begin{equation}\label{eqn: TrQg Ball}
    Q_g(E^\perp,E^\perp)=\delta^2\Sigma(E^\perp,E^\perp) = -k\int_\Sigma|E^\perp|^2_gd\mu.
\end{equation}

For a generic compact free boundary submanifold $\Sigma'\subset B^n$, this formula does not necessarily hold. However, \cref{sum} with $\Omega=B^n$ recovers \cref{eqn: TrQg Ball} traced over constant vector fields of $\mathbb R^{n}$. Indeed, for $\Sigma'$ generic, \cref{sum} with $\Omega=B^n$ implies 
\[\sum_{l=1}^n Q_g(E_l^\perp,E_l^\perp) = -(n-k)\cdot vol_{k-1}(\partial\Sigma'),\]
while, for $\Sigma$ minimal, \eqref{eqn: TrQg Ball} gives
\[\sum_{l=1}^n Q_g(E_l^\perp,E_l^\perp) = -k(n-k)\cdot vol_{k}(\Sigma)=-(n-k)\cdot vol_{k-1}(\partial\Sigma),\]
where the last equality follows from \cite[Proposition 2.4]{LiFreeBoundary}. This observation is analogous to the one described in \cite[Theorem 5.1]{FranzTrinca} for the round sphere.
\end{remark}

Given \cref{Q,sum}, it is natural to consider rescaled constant vector fields of $\mathbb R^n$ as competitors for the stability of minimal submanifolds in conformal domains.

\begin{lemma}\label{ST}
    Let $(\Omega,g)$ be a bounded domain of $\Rbb^n,$ and let $\tg$ be a Riemannian metric conformal to $g$ of conformal factor $e^{2u}.$ Let $\Sigma$ be a $k$-dimensional free boundary minimal submanifold of $(\Omega,\tg)$. Then, respectively at each point of $\Sigma$ and $\partial\Sigma$ we have
    \begin{align*}
        e^{2u}\sum_{\ell=1}^n S_{\tg}(\tE_\ell^\perp,\tE_\ell^\perp)&= k|\nabla^\perp u|^2_g - e^{2u} K_{\tg}(T\Sigma,N\Sigma)\\
        e^u\sum_{\ell=1}^n T_{\tg}(\tE_\ell^\perp,\tE_\ell^\perp) &= -(n-k)\nu_g(u) + \sum_{\ell=1}^ng(\alpha_{\partial \Omega}(E_\ell^\perp,E_\ell^\perp), \nu_g)
    \end{align*}
    where $\{E_1, \ldots, E_n\}$ is an arbitrary orthonormal basis of $\Rbb^n,$ $\tE_\ell = e^{-u}E_\ell$ for $\ell = 1,\ldots,n,$ and $\alpha_{\partial\Omega}$ is the second fundamental form of $\partial\Omega$ in $\Rbb^n.$ Here, $K_{\tg}(T\Sigma,N\Sigma)$ denotes
$$K_{\tg}(T\Sigma,N\Sigma)=\sum_{i=1}^k\sum_{r=k+1}^n K_{\tg}(\tv_i,\tv_r),$$
where $\{\tv_1,\dots,\tv_k,\tv_{k+1},\dots,\tv_{n}\}$ is any orthonormal frame of $(\Omega,\tg)$ such that $\tv_1,\ldots,\tv_k$ are tangent to $\Sigma$ and $\tv_{k+1},\dots,\tv_{n}$ are normal to $\Sigma$.
\end{lemma}
\begin{proof}
    Given $x\in\Sigma$, let $\{v_1, \ldots, v_n\}$ be a local orthonormal frame with respect to $g$ such that $v_1, \ldots, v_k$ are tangent to $\Sigma$ and $v_{k+1},\ldots,v_n$ are normal to $\Sigma$. This local frame induces an orthonormal frame with respect to $\tg$ with the same properties via $\tv_\ell=e^{-u}v_\ell$. Since $\sum_{\ell=1}^n S_{\tg}(\tE_\ell^\perp,\tE_\ell^\perp)$ and $\sum_{\ell=1}^n T_{\tg}(\tE_\ell^\perp,\tE_\ell^\perp)$ are independent of the orthonormal basis $\{E_1, \ldots, E_n\}$, we can assume that $\tE_\ell(x)=\tv_\ell(x)=e^{-u}v_\ell(x)$ for all $\ell=1,\ldots,n$.
    
    Combining \cref{Q,sum}, it is straightforward to verify that
    \[e^{2u}\sum_{\ell=1}^n S_{\tg}(\tE_\ell^\perp,\tE_\ell^\perp) = -(n-k)|\nabla^\top u|^2_g + (n-k)\divv_\Sigma(\nabla u) + k(n-k)|\nabla u|^2_g + k\sum_{r=k+1}^n\nabla^2u(v_r,v_r).\]
    Furthermore, Item (iii) of \cref{conformallemma} implies that for all $\ell,m = 1, \ldots, n$ with $\ell \neq m$, we have
    \[e^{2u} K_{\tg}(E_\ell, E_m) = {E_\ell(u)}^2+ {E_m(u)}^2- |\nabla u|^2_g  - \nabla^2 u (E_\ell, E_\ell) - \nabla^2 u (E_m, E_m),\]
    from which we deduce
    \begin{equation}\label{eqn: TraceTN}
    \begin{aligned}
        &e^{2u}\sum_{i=1}^k\sum_{r=k+1}^nK_{\tg}(\tE_i,\tE_r) = e^{2u}\sum_{i=1}^k\sum_{r=k+1}^nK_{\tg}(E_i,E_r)\\
        &= (n-k)|\nabla^\top u|^2_g + k|\nabla^\perp u|^2_g - k(n-k)|\nabla u|^2_g - (n-k)\divv_\Sigma(\nabla u) - k \sum_{r=k+1}^n\nabla^2u(E_r,E_r).
    \end{aligned}
    \end{equation}
    In particular, at $x,$ we obtain the desired formula
    \[\sum_{\ell=1}^nS_{\tg}(\tE_\ell^\perp, \tE_\ell^\perp) = - \sum_{i=1}^k\sum_{r=k+1}^n K_{\tg}(\tv_i, \tv_r) + ke^{-2u}|\nabla^\perp u|^2_g.\]
    
    Using again \cref{Q,sum}, but now at $x\in\partial\Sigma$, we see that:
    \[e^u \sum_{\ell=1}^nT_{\tg}(\tE_\ell^\perp,\tE_\ell^\perp) = \sum_{\ell=1}^n T_g(E_\ell^\perp,E_\ell^\perp) - (n-k)\nu_g(u) = \sum_{\ell=1}^ng(\alpha_{\partial \Omega}(E_\ell^\perp,E_\ell^\perp), \nu_g) - (n-k)\nu_g(u)\]
    from which we conclude.
\end{proof}

We now show how to control $\sum_{\ell=1}^{n} \int_{\Sigma} S_{\tilde{g}}(\tE_\ell^{\perp}, \tE_\ell^{\perp}) d\tilde{\mu}$ in terms of boundary elements.

\begin{lemma}\label{inequality}
Let $(\Omega,g)$ be a bounded domain of $\Rbb^n$, and let $\tg$ be a Riemannian metric conformal to $g$ of conformal factor $e^{2u}$ with non-negative sectional curvatures. Let $\Sigma$ be a compact $k$-dimensional free boundary minimal submanifold of $(\Omega,\tg)$ with $2 \leq k \leq n-2.$ Then
\begin{equation*}
\sum_{\ell=1}^{n} \int_{\Sigma} S_{\tilde{g}}(\tE_\ell^{\perp}, \tE_\ell^{\perp}) d\tilde{\mu} \leq 2 \int_{\partial \Sigma} \nu_{\tg}(u) d\tilde{a},
\end{equation*}
where $\{E_1, \ldots, E_n\}$ is an arbitrary orthonormal basis of $\Rbb^n,$ and $\tE_\ell = e^{-u}E_\ell$ for every $\ell=1,\ldots,n$. Moreover, if the conformal metric $g$ has positive sectional curvature, then the inequality is strict.
\end{lemma}
\begin{proof}
Given $x\in\Sigma$, let $\{v_1, \ldots, v_n\}$ be a local orthonormal frame with respect to $g$ such that $v_1, \ldots, v_k$ are tangent to $\Sigma$ and $v_{k+1},\ldots,v_n$ are normal to $\Sigma$. Since $\sum_{\ell=1}^n S_{\tg}(\tE_\ell^\perp,\tE_\ell^\perp)$ is independent of the orthonormal basis $\{E_1, \ldots, E_n\}$, we can assume that $E_\ell(x)=v_\ell(x)$ for all $\ell=1,\ldots,n$.

As a first step, we observe that \eqref{eqn: TraceTN} can be rewritten in the following form
\begin{equation*}
    \begin{aligned}
        e^{2u} \sum_{i=1}^{k} \sum_{r=k+1}^{n}K_{\tilde{g}}(E_{i}, E_r) = -(k-1)(n-k)&|\nabla^{\top} u|_g^{2} - k(n-k-1)|\nabla^{\perp} u|_g^{2}\\
        & - (n-k) \divv_{\Sigma}(\nabla u)- k \sum_{r=k+1}^{n} \nabla^2{u}(E_r, E_r).
    \end{aligned}
\end{equation*}
Moreover, by minimality of $\Sigma$ with respect to $\tg,$ one can easily verify that $\divv_{\Sigma}(\nabla u) = -g(\nabla u,H)+\divv_{\Sigma}(\nabla^{\top} u)= -k|\nabla^{\perp}u|^2_g+ \divv_{\Sigma}(\nabla^{\top} u),$ and thus
\begin{equation}\label{tangent/tangent}
\begin{aligned}
e^{2u} \sum_{i=1}^{k} \sum_{r=k+1}^{n}K_{\tilde{g}}(E_{i}, E_r) = -(k-1)(n-k)&|\nabla^{\top} u|_g^{2} +k|\nabla^{\perp} u|_g^{2}\\
& - (n-k) \divv_{\Sigma}(\nabla^{\top} u)- k \sum_{r=k+1}^{n} \nabla^2{u}(E_r, E_r).
\end{aligned}
\end{equation}
Similarly, as $k\leq n-2$, we can sum over $r,s=k+1, \ldots, n,$, and obtain
\[e^{2u} \sum_{r\neq s} K_{\tilde{g}}(E_r, E_s) = -(n-k-1)\bigg((n-k)|\nabla^{\top} u|_g^{2} + (n-k-2)|\nabla^{\perp} u|_g^{2} + 2\sum_{r=k+1}^{n}\nabla^2u(E_r, E_r)\bigg),\]
which, in particular, implies
\begin{equation}\label{eqn: NN}
\sum_{r=k+1}^n \nabla^2{u}(E_r, E_r) = -\frac{1}{2}\bigg((n-k)|\nabla^{\top}u|^{2}_g  + (n-k-2)|\nabla^{\perp} u|^{2}_g + \frac{e^{2u}}{n-k-1} \sum_{r\neq s} K_{\tilde{g}}(E_r, E_s)\bigg).
\end{equation}
Inserting \eqref{eqn: NN} into \eqref{tangent/tangent}, we can directly compute 
\begin{align*}
e^{2u}\sum_{i=1}^{k} \sum_{r=k+1}^{n}K_{\tilde{g}}(E_{i}, E_r) & =\frac{1}{2}\left(-(k-2)(n-k)|\nabla^{\top} u|_g^{2} + k(n-k)|\nabla^{\perp} u|_g^{2}\right)\\
&\quad \quad \quad \quad - (n-k) \divv_{\Sigma}(\nabla^{\top} u) + \frac{k e^{2u}}{2(n-k-1)} \sum_{r \neq s} K_{\tilde{g}}(E_r, E_s).
\end{align*}
Isolating $k|\nabla^{\perp} u|^{2}$ in the above equality, we can use its expression in \cref{ST} to obtain
\begin{align*}
    \sum_{\ell=1}^nS_{\tg}(\tE_\ell^\perp,\tE_\ell^\perp) &= e^{-2u}\left((k-2)|\nabla^\top u|^2_g + 2\divv_\Sigma(\nabla^\top u)\right)\\
    &\quad -\frac{1}{n-k}\bigg((n-k-2)\sum_{i=1}^k\sum_{r=k+1}^nK_{\tg}(v_i,v_r) + \frac{k}{n-k-1}\sum_{r\neq s}K_{\tg}(v_r,v_s)\bigg).
\end{align*}
Observe that the second term does not depend on the choice of the orthonormal frame and that the first term is in divergence form, indeed, $$\divv_\Sigma\big(e^{(k-2)u}\nabla^\top u\big) = (k-2)e^{(k-2)u}|\nabla^\top u|^2_g + e^{(k-2)u}\divv_\Sigma(\nabla^\top u).$$  
Finally, since the sectional curvatures of $\tg$ are non-negative and $d\tmu = e^{ku}d\mu,$ the divergence theorem implies
\begin{align*}
    \sum_{i=1}^{n} \int_{\Sigma} S_{\tilde{g}}(\tE_i^{\perp}, \tE_i^{\perp}) d\tilde{\mu} \leq  2\int_{\Sigma}\operatorname{div}_{\Sigma}\big(e^{(k-2)u}\nabla^{\top} u\big) d\mu = 2\int_{\partial\Sigma}e^{(k-2)u}\nu_g(u) da = 2\int_{\partial\Sigma}\nu_{\tg}(u) d\tilde{a},
\end{align*}
where, in the first inequality, we use that $2\leq k\leq n-2$, that $K_{\tg} \geq 0$ and that $(k-2)|\nabla^\top u|^2_g\leq 2(k-2) |\nabla^\top u|^2_g$. The last equality follows from $\nu_{\tg}(u)=e^{-u}\nu_{g}(u)$ and $d\tilde{a}=e^{(k-1)u}da$.
\end{proof}

We can finally use the computations that we have carried out to prove our main results: \cref{maintheorem} and \cref{maincor}. 
\begin{proof}[Proof of \cref{maintheorem} and \cref{maincor}]\label{proof}

Let $\Sigma$ be a free boundary minimal submanifold of $(\Omega, \tg)$ with conformal factor $e^{2u}$. All we need to show is that $\tdelta^2\Sigma(X,X) = Q_{\tg}(X,X) < 0$ for some vector field $X\in\Gamma(N\Sigma)$, where $\tdelta^2\Sigma$ is the second variation of $\Sigma$ as a submanifold of $ (\Omega, \tg)$. Obviously, it is enough to show that:
\[
\sum_{\ell=1}^n  \tdelta^2\Sigma(\tE_\ell,\tE_\ell)<0,
\]
where $\{E_1,\ldots,E_n\}$ is an arbitrary orthonormal basis of $\mathbb R^n$ and $\tE_\ell = e^{-u}E_\ell$.

We first prove \cref{maintheorem}. Combining \cref{ST,inequality}, we obtain
\begin{equation}\label{ineq}
            \sum_{\ell=1}^n\tdelta^2\Sigma(E_\ell^\perp,E_\ell^\perp) \leq -(n-k-2)\int_{\partial\Sigma} \nu_{\tg}(u)d\ta + \sum_{\ell=1}^n\int_{\partial\Sigma}e^{-u}g(\alpha_{\partial\Omega}(E_\ell^\perp,E_\ell^\perp),\nu_g)d\ta.
\end{equation}
Now, we can estimate the first term using \cref{ST} and the $(n-k)$-convexity of $\partial\Omega$ with respect to $\tg$ (note that $p\leq n-k$, hence $p$-convexity implies $(n-k)$-convexity). Indeed, these give
        \begin{equation}\label{ineq2}
            0 \geq \sum_{\ell=1}^n\tg(\tnabla_{\tE_\ell^\perp}\tE_\ell^\perp,\nu_{\tg}) = \sum_{\ell=1}^nT_{\tg}(\tE_\ell^\perp,\tE_\ell^\perp) = - (n-k)\nu_{\tg}(u) +  \sum_{\ell=1}^n e^{-u}g(\alpha_{\partial\Omega}(E_\ell^\perp,E_\ell^\perp),\nu_{g}),
        \end{equation}
    which implies
        \[\sum_{\ell=1}^n\tdelta^2\Sigma(E_\ell^\perp,E_\ell^\perp) \leq \frac{2}{n-k}\sum_{\ell=1}^n\int_{\partial\Sigma}e^{-u}g(\alpha_{\partial\Omega}(E_\ell^\perp,E_\ell^\perp),\nu_g)d\ta \leq 0.\]
    If $\tg$ has positive sectional curvature or $\partial\Omega$ is strictly $(n-k)$-convex with respect to $g$ or $\tg$, then the previous inequality is strict. This concludes the proof of \cref{maintheorem}.

We now prove \cref{maincor}. If $\partial\Omega$ is $(n-k)$-convex with respect to $g$ and $\nu_{\tg}(u)=e^{-u}\nu_{g}(u)>0$, it is straightforward to see from \eqref{ineq2} that $\partial\Omega$ is strictly $(n-k)$-convex with respect to $\tg$. In the same way, if $\partial\Omega$ is $(n-k)$-convex with respect to $\tg$ and $\nu_{\tg}(u)=e^{-u}\nu_{g}(u)<0$, then $\partial\Omega$ is strictly $(n-k)$-convex with respect to $g$.
\end{proof} 

\bibliographystyle{plain}
\bibliography{references}

\begin{thebibliography}{10}

\bibitem{AmbrozioCarlottoSharp2018}
Lucas Ambrozio, Alessandro Carlotto, and Ben Sharp.
\newblock Index estimates for free boundary minimal hypersurfaces.
\newblock {\em Math. Ann.}, 370(3-4):1063--1078, 2018.

\bibitem{Besse}
Arthur~L. Besse.
\newblock {\em Einstein manifolds}, volume~10 of {\em Ergeb. Math. Grenzgeb., 3. Folge}.
\newblock Springer, Cham, 1987.

\bibitem{CecilRyan2015}
Thomas~E. Cecil and Patrick~J. Ryan.
\newblock {\em Geometry of hypersurfaces}.
\newblock Springer Monographs in Mathematics. Springer, New York, 2015.

\bibitem{chen2023conformal}
Hang Chen.
\newblock A conformal invariant and its application to the nonexistence of minimal submanifolds.
\newblock {\em arXiv preprint arXiv:2310.09724}, 2023.

\bibitem{MisteliFranz2024}
Santiago Cordero-Misteli and Giada Franz.
\newblock Estimating the {M}orse index of free boundary minimal hypersurfaces through covering arguments.
\newblock {\em J. Reine Angew. Math.}, 807:187--220, 2024.

\bibitem{FranzTrinca}
Giada Franz and Federico Trinca.
\newblock On the stability of minimal submanifolds in conformal spheres.
\newblock {\em J. Geom. Anal.}, 33(10):16, 2023.
\newblock Id/No 335.

\bibitem{Fraser2000}
Ailana Fraser.
\newblock On the free boundary variational problem for minimal disks.
\newblock {\em Comm. Pure Appl. Math.}, 53(8):931--971, 2000.

\bibitem{FraserIndex}
Ailana Fraser.
\newblock Index estimates for minimal surfaces and {$k$}-convexity.
\newblock {\em Proc. Amer. Math. Soc.}, 135(11):3733--3744, 2007.

\bibitem{FraserLi}
Ailana Fraser and Martin Man-chun Li.
\newblock Compactness of the space of embedded minimal surfaces with free boundary in three-manifolds with nonnegative {R}icci curvature and convex boundary.
\newblock {\em J. Differential Geom.}, 96(2):183--200, 2014.

\bibitem{FraserSchoen2011}
Ailana Fraser and Richard Schoen.
\newblock The first {S}teklov eigenvalue, conformal geometry, and minimal surfaces.
\newblock {\em Adv. Math.}, 226(5):4011--4030, 2011.

\bibitem{FraserSchoenMinimal}
Ailana Fraser and Richard Schoen.
\newblock Minimal surfaces and eigenvalue problems.
\newblock In {\em Geometric analysis, mathematical relativity, and nonlinear partial differential equations}, volume 599 of {\em Contemp. Math.}, pages 105--121. Amer. Math. Soc., Providence, RI, 2013.

\bibitem{FraserSchoenSharp}
Ailana Fraser and Richard Schoen.
\newblock Sharp eigenvalue bounds and minimal surfaces in the ball.
\newblock {\em Invent. Math.}, 203(3):823--890, 2016.

\bibitem{HuWei2003}
Ze-Jun Hu and Guo-Xin Wei.
\newblock On the nonexistence of stable minimal submanifolds and the {L}awson-{S}imons conjecture.
\newblock {\em Colloq. Math.}, 96(2):213--223, 2003.

\bibitem{LawsonSimons}
H.~Blaine Lawson, Jr. and James Simons.
\newblock On stable currents and their application to global problems in real and complex geometry.
\newblock {\em Ann. of Math. (2)}, 98:427--450, 1973.

\bibitem{LiFreeBoundary}
Martin Man-chun Li.
\newblock Free boundary minimal surfaces in the unit ball: recent advances and open questions.
\newblock In {\em Proceedings of the {I}nternational {C}onsortium of {C}hinese {M}athematicians 2017}, pages 401--435. Int. Press, Boston, MA, [2020] \copyright 2020.

\bibitem{Lima2022}
Vanderson Lima.
\newblock Bounds for the {M}orse index of free boundary minimal surfaces.
\newblock {\em Asian J. Math.}, 26(2):227--252, 2022.

\bibitem{LimaMenezes}
Vanderson Lima and Ana Menezes.
\newblock Eigenvalue problems and free boundary minimal surfaces in spherical caps.
\newblock {\em arXiv preprint arXiv:2307.13556}, 2023.

\bibitem{Medvedev}
Vladimir Medvedev.
\newblock On free boundary minimal submanifolds in geodesic balls in $\mathbb{H}^n$ and $\mathbb{S}^n_+$.
\newblock {\em arXiv preprint arXiv:2311.02409}, 2023.

\bibitem{MooreSchulteIndex}
John~Douglas Moore and Thomas Schulte.
\newblock Minimal disks and compact hypersurfaces in {E}uclidean space.
\newblock {\em Proc. Amer. Math. Soc.}, 94(2):321--328, 1985.

\bibitem{Sargent2017}
Pam Sargent.
\newblock Index bounds for free boundary minimal surfaces of convex bodies.
\newblock {\em Proc. Amer. Math. Soc.}, 145(6):2467--2480, 2017.

\bibitem{Schoen}
Richard Schoen.
\newblock Minimal submanifolds in higher codimension.
\newblock {\em Mat. Contemp.}, 30:169--199, 2006.

\bibitem{ShenHe2001}
Yi-Bing Shen and Qun He.
\newblock On stable currents and positively curved hypersurfaces.
\newblock {\em Proc. Amer. Math. Soc.}, 129(1):237--246, 2001.

\bibitem{ShenXu2000}
Yi-Bing Shen and Hui-Qun Xu.
\newblock On the nonexistence of stable minimal submanifolds in positively pinched {R}iemannian manifolds.
\newblock In {\em Geometry and topology of submanifolds, {X} ({B}eijing/{B}erlin, 1999)}, pages 274--283. World Sci. Publ., River Edge, NJ, 2000.

\bibitem{Simons1968}
James Simons.
\newblock Minimal varieties in riemannian manifolds.
\newblock {\em Ann. of Math. (2)}, 88:62--105, 1968.

\bibitem{Takagi1975}
Ryoichi Takagi.
\newblock Real hypersurfaces in a complex projective space with constant principal curvatures.
\newblock {\em J. Math. Soc. Japan}, 27:43--53, 1975.

\end{thebibliography}
\end{document}